\documentclass[a4paper,12pt]{article}

\usepackage{amssymb, amsmath, latexsym, amsthm, cite}

\newcommand{\Q}{\mathbb{Q}}

\newcommand{\F}{\mathbb{F}}

\newcommand{\N}{\mathbb{N}}

\renewcommand{\b}[1]{{\bf #1}}

\newcommand{\x}{\b{x}}

\newcommand{\Z}{\mathbb{Z}}

\newtheorem{theorem}{Theorem}

\newtheorem{lemma}{Lemma}

\author{Jahan Zahid}

\date{}

\title{Non-singular points on Hypersurfaces over $\F_{q}$}

\begin{document}

\maketitle

\section{Introduction}

In this survey we investigate hypersurfaces defined over finite fields $\F_{q}$. More specifically we wish to determine for which hypersurfaces one can ensure the existence of a non-singular point taking the cardinality $q$ of our ambient field large if need be. Additionally for such hypersurfaces we will find a lower bound on $q$ for which a non-singular point is guaranteed. 

We shall assume that our hypersurface is projective and is defined by the homogeneous polynomial
\[ F(x_{1},\ldots,x_{n}) \in \F_{q}[x_{1},\ldots,x_{n}], \]
in $n \ge 3$ variables. Recall that a non-singular point $\x \in \F_{q}^{n}$ is one in which 
\[ \nabla F(\x) := \bigg(\frac{\partial F}{\partial x_{1}}(\x),\ldots,\frac{\partial F}{\partial x_{n}}(\x) \bigg) \ne \b{0}. \]

\subsection{Motivation} 

It is a classical theorem of Chevalley \& Warning that every form $F$ of degree $d$ in $n>d$ variables has a non-trivial zero in $\F_{q}$. However there seems to be much less in the literature giving conditions which ensure that $F$ has a non-singular point. This question is worth addressing for a number of reasons. For example suppose we have a form $\mathfrak{F}$ defined over the $p$-adic numbers $\Q_{p}$, we may wish to determine whether or not $\mathfrak{F}$ has a non-trivial zero. This is of course a requisite for the corresponding form to have a non-trivial zero in the rational numbers $\Q$. A standard and well known method for determining the existence of a $p$-adic point is by an application of Hensel's lemma viz. assume without loss of generality $\mathfrak{F}$ is defined over $\Z_{p}$, then if there is some $\x$ over $\Z_{p}$ such that 
\[ \mathfrak{F}(\x) \equiv 0 \pmod{p} \qquad \mbox{and} \qquad \nabla \mathfrak{F}(\x) \not\equiv \b{0} \pmod{p}  \]
then $\mathfrak{F}$ has a non-trivial $p$-adic point. 

A conjecture of Artin \cite[Preface]{Art65} states that every $p$-adic form of degree $d$ in $n> d^{2}$ variables has a non-trivial zero. There exist forms in $n=d^{2}$ variables with only the trivial zero in every $p$-adic field, so this is the best we can hope for. The conjecture is false in general (see for example \cite{Ter66}). However by a rather remarkable theorem of Ax \& Kochen \cite{AxKoc65} it is known to be true for any fixed degree $d$, provided the characteristic of the residue class field is large enough. A reasonable lower bound for the characteristic required is only known in a handful of cases. For example, recently Wooley \cite{Woo08} has shown that for $d=7$ and $11$ we require the size of the residue class field to respectively exceed $883$ and $8053$. These estimates depend on the ability to ascertain the existence of a non-singular point on the $\F_{q}$-varieties under consideration. It is often in this regard that we are motivated to seek non-singular points on varieties.

The existence of a $p$-adic point is also crucial for applications of the Hardy-Littlewood circle method where one requires a certain ``singular series'' to be positive. The singular series of some function is positive precisely when that function has a non-singular $p$-adic point for every $p$.

\subsection{Varieties with a non-singular point} 

We necessarily require some kind of classification of our variety $F(\x)=0$, since there are examples such as $F(\x) = G(\x)^{2}$ with only singular zeros. After factorizing $F$ over the algebraic closure $\bar{\F}_{q}$ we may write
\begin{equation} \label{factors}
F(\x)= F_{1}(\x)^{a_{1}} \cdots F_{r}(\x)^{a_{r}} 
\end{equation}
where $a_{i} \in \N$ and each $F_{i}$ is absolutely irreducible and pairwise coprime. If $a_{i} \ge 2$ for each $i$, then it is clear that $F$ has only singular points. We may therefore assume henceforth that $a_{1} = 1$. Consequently we can write 
\[ F(\x) = G(\x)H(\x) \]
where $G$ is absolutely irreducible and $G$ does not divide $H$. Since we are interested in finding non-singular points over the base field $\F_{q}$, for simplicity we shall assume that both $G$ and $H$ are defined over $\F_{q}$. 

It is clear that we can find a non-singular zero of $F$, provided that we can find a non-singular zero of $G$ which is not a zero of $H$. Since for any such point $\x \in \F_{q}^{n}$, clearly $F(\x)=0$ and
\[ \nabla F(\x) = H(\x)\nabla G(\x) \ne \b{0}. \]
This exact question has been studied by Lewis \& Schuur \cite[Theorem 1]{LewScr73} who prove the following result in their paper.

\begin{theorem}[Lewis \& Schuur, 1973] \label{lewscr73}
Let $G$ be an absolutely irreducible form of degree $d$ over $\F_{q}$ and $H$ a form of degree $e$ over $\F_{q}$ not divisible by $G$. Then there exists a function $A(d,e)$ such that if $\F_{q}$ is a finite field of cardinality $q > A(d,e)$, then there exists an $\F_{q}$-point which is a non-singular zero of $G$ and which is not a zero of $H$.
\end{theorem}

We shall apply a recent result of Cafure \& Matera \cite{CafMat06} in order to find a permissible value for $A(d,e)$ in the Lewis \& Schuur theorem. More precisely we will prove 

\begin{theorem} \label{thm2}
Let $G$ be an absolutely irreducible form of degree $d$ over $\F_{q}$ and $H$ a form of degree $e$ over $\F_{q}$ not divisible by $G$. Then provided that
\[ q > \tfrac{1}{4}\big(\alpha + \sqrt{\alpha^{2}+4\beta})^{2} \]
where
\[ \alpha = (d-1)(d-2) \quad \mbox{and} \quad \beta = 5d^{13/3}+d(d+e-1) \]
there is a non-singular zero of $G$ which is not a zero of $H$.
\end{theorem}

When the form $F$ is absolutely irreducible we can do slightly better, by applying an effective version of the first Bertini theorem, along with the sharper estimates which are available for point counting on curves. For us the first Bertini theorem says that given any absolutely irreducible hypersurface, there exists a linear variety of dimension $2$ whose intersection with the former is absolutely irreducible. The existence of this linear variety has been made quantitative by Kaltofen \cite{Kal95} which allows us to prove

\begin{theorem} \label{thm3}
Let $F$ be an absolutely irreducible form of degree $d$ over $\F_{q}$ then provided that,
\[ q > \tfrac{1}{2}(3d^{4}-4d^{3}+5d^{2}) \]
there is a non-singular zero of $F$.
\end{theorem}

We may also apply a flexible version of the Bertini theorem (as given in  \cite[Corollary 3.4]{CafMat06}) in the general case i.e.\ whenever $F(\x) = G(\x)H(\x)$. However there is more delicacy involved dealing with the degrees of each irreducible factor $F_{i}$ in (\ref{factors}). In this respect it seems difficult to asymptotically improve on the permissible lower bound for $q$, given in Theorem \ref{thm2}. One may wish to read \cite{Woo08} for further details.

 \newpage

\section{Some preliminaries}

In this section we shall review some of the elements which are available for point counting on varieties. 

\subsection{Lang--Weil estimates}

Arguably the genesis for much subsequent work on estimating the number of points on $\F_{q}$-varieties, is due to a theorem of Lang \& Weil \cite{LanWei54}.

\begin{theorem}[Lang--Weil, 1954]
Given any absolutely irreducible variety $V$ of degree $d$ and dimension $r$ in $n$-dimensional projective space over $\F_{q}$, there exists a constant $C = C(n,r,d)$ such that the number of $\F_{q}$-points $N$ of $V$ satisfies
\[ | N - q^{r} | \le (d-1)(d-2)q^{r-1/2}+Cq^{r-1}. \]
\end{theorem}

In the case of whenever the variety $V$ is defined by a hypersurface of degree $d$ in $n$ variables over $\F_{q}^{n}$, there has been much work on finding a permissible value for the constant $C= C(n,r,d)$ in the Lang \& Weil theorem. For example Schmidt \cite[Theorem 5A, p.210]{Sch76} proved that we can take 
\begin{equation} \label{schC}
 C = 6d^{2}k^{2^{k}}, 
 \end{equation}
where $k = \tfrac{1}{2}d(d+1)$. More recently Cafure \& Matera \cite[Theorem 5.2]{CafMat06} proved the following

\begin{theorem}[Cafure \& Matera, 2006] \label{cafmat}
The number of points $N$, on an absolutely irreducible $\F_{q}$ hypersurface in $n$ variables of degree $d$ satisfies
\[ | N - q^{n-1} | \le (d-1)(d-2)q^{n-3/2} + 5d^{13/3}q^{n-2}. \]
\end{theorem}

This substantial improvement of (\ref{schC}) owes much to the work of Kaltofen \cite{Kal95} for providing an effective version of the first Bertini theorem. We shall discuss this further in the next section.

\subsection{Upper bounds on the number of $\F_{q}$-points} 

We will also require bounding from above, the number of points on an affine variety of prescribed degree. Thus note the following well known result (for example see \cite[Lemma 2.1]{CafMat06}).

\begin{lemma} \label{uppbd1}
Any $\F_{q}$-variety $V$, of dimension $r$ and degree $d$ in $\F_{q}^{n}$ satisfies
\[ \#(V \cap \F_{q}^{n}) \le dq^{r}. \]
\end{lemma}

In order to keep some control on the number of singular points on our variety we need to estimate the number of points on the intersection of two hypersurfaces, with no common component. To prove the next lemma it will be necessary to utilize the $B\acute{e}zout \; inequality$ (for example see \cite[p.148]{Ful98}): if $V$ and $W$ are affine varieties in $\F_{q}^{n}$ then
\begin{equation} \label{bez}
\deg(V \cap W) \le \deg V \deg W.
\end{equation}

\begin{lemma} \label{uppbd2}
Let $F_{1},F_{2} \in \F_{q}[\x]$ be non-zero polynomials of degrees $d_{1}$ and $d_{2}$ without a common factor in $\bar{\F}_{q}[\x]$ then the variety $V$ defined by $F_{1},F_{2}$ satisfies
\[ \#(V \cap \F_{q}^{n}) \le d_{1}d_{2}q^{n-2}. \]
\end{lemma}

\begin{proof}
Since $F_{1}$ and $F_{2}$ have no common factors in $\bar{\F}_{q}[\x]$, then $V$ is an $\F_{q}$ variety of dimension $n-2$. From the $B\acute{e}zout \; inequality$ (\ref{bez}), $\deg V \le d_{1}d_{2}$. The result then follows by Lemma \ref{uppbd1}.
\end{proof}

\subsection{Non-singular points on curves}

Finally we will also need to estimate the number of non-singular points on an absolutely irreducible curve. The following lemma is due to the work of Leep \& Yeomans \cite{LeeYeo94}. 

\begin{lemma} \label{leeyeo}
Let $P \in \F_{q}[x,y]$ be an absolutely irreducible polynomial of degree $d$. Then the number $N'$ of non-singular zeros of $P$ satisfies
\begin{eqnarray*}
N' \ge q + 1 - \tfrac{1}{2} (d-1)(d-2) [2\sqrt{q}], 
\end{eqnarray*}
where $[\gamma]$ denotes the least integer not exceeding $\gamma$.
\end{lemma}

\begin{proof}
Write $S$ to denote the number of $\F_{q}$ singular zeros of $P$. Then if the curve defined by $P(x,y)=0$ has genus $g$, it follows from \cite[Corollary 1]{LeeYeo94} that
\[ | N' + S -(q+1)| \le g([2\sqrt{q}]-1) + \tfrac{1}{2}(d-1)(d-2). \]
Next we use the above estimate together with the following bound on the genus
\[ g \le \tfrac{1}{2}(d-1)(d-2)-S, \] 
which comes from \cite[Lemma 1]{LeeYeo94}, to obtain the required bound
\[ N' \ge q+1 - \tfrac{1}{2}(d-1)(d-2)[2\sqrt{q}]. \]
\end{proof}

\section{Proof of Theorems \ref{thm2} and \ref{thm3}}

Our strategy will be relatively straight forward: to prove Theorem \ref{thm2} we shall bound from below the total number of points on the hypersurface $G(\x)=0$, whilst simultaneously bounding from above, the number of singular points of $G(\x)=0$ plus  the number of points on the intersection $G(\x)=H(\x)=0$; for Theorem \ref{thm3} we shall use an effective version of the first Berini theorem and apply Lemma \ref{leeyeo} to obtain the desired conclusion.

\begin{proof}[Proof of Theorem \ref{thm2}]

If we let
\[ S_{1} = \# \{ \x \in \F_{q}^{n} : G(\x)=0 \; \;  and \; \; \nabla G(\x) = \b{0} \} \]
and 
\[ S_{2} = \# \{ \x \in \F_{q}^{n} : G(\x)=0 \; \;  and \; \; H(\x) = 0 \}, \]
then it follows that we have a non-singular point of $G$ which not a zero of $H$ provided that $N - S_{1}-S_{2}>0$, where $N$ denotes the number of points of $G$.

If $\nabla G$ is identically zero then it follows that $G(x_{1},\ldots,x_{n}) = M(x_{1}^{p},\ldots,x_{n}^{p})$, where $p$ denotes the characteristic of $\F_{q}$, for some form $M$. Consequently over $\bar{\F}_{q}[\x]$, $G(\x)$ factors into $M'(x_{1},\ldots,x_{n})^{p}$ for some form $M'$, contradicting that $G$ is absolutely irreducible. Therefore some component of $\nabla G$ is non-zero, $\frac{\partial G}{\partial x_{1}}$ say. 

It is clear that $G$ and $\frac{\partial G}{\partial x_{1}}$ have no common factor in $\bar{\F}_{q}[\x]$, therefore by Lemma \ref{uppbd2} it follows that $S_{1} \le d(d-1)q^{n-2}$. Moreover since $G$ and $H$ do not have a common component it also follows that $S_{2} \le deq^{n-2}$. Hence by Theorem \ref{cafmat} 
\[ N - S_{1} - S_{2} \ge  q^{n-1} - (d-1)(d-2)q^{n-3/2} - (5d^{13/3} + d(d-1) + de)q^{n-2}. \]
It is clear that $N -S_{1} - S_{2}>0$ provided that
\[  q - (d-1)(d-2)\sqrt{q} - 5d^{13/3} - d(d+e-1)> 0, \]
from which the conclusion follows.
\end{proof}

Before proving Theorem \ref{thm3}, we need to introduce some notation. Let $L$ be a field and consider a polynomial $f \in L[x_{0},x_{1},\ldots,x_{n}]$. When $\b{\xi} \in L^{3n+1}$, we write $f|_{\b{\xi}}=f|_{\b{\xi}}(X,Y)$ to denote the sliced polynomial
\[ f(\xi_{0}+X,\xi_{1}+\xi_{n+1}X+\xi_{2n+1}Y,\ldots,\xi_{n}+\xi_{2n}X+\xi_{3n}Y). \]

Next note the following result of Kaltofen \cite[p.285]{Kal95} which is made explicit by Cafure \& Matera \cite[Corollary 3.2]{CafMat06}.

\begin{lemma} \label{bert}
Let $f \in \F_{q}[x_{0},\ldots,x_{n}]$ be an absolutely irreducible polynomial of degree $d \ge 2$. Then the number of slices $\b{\xi} \in \F_{q}^{3n+1}$, for which the polynomial $f|_{\b{\xi}}$ is not absolutely irreducible is at most $\frac{1}{2}(3d^{4}-4d^{3}+5d^{2})q^{3n}$.
\end{lemma}

\begin{proof}[Proof of Theorem \ref{thm3}]

We know that $F$ is absolutely irreducible, so by Lemma \ref{bert} provided that
\[ q > \tfrac{1}{2}(3d^{4}-4d^{3}+5d^{2}), \] 
there exists a slice $\xi \in \F_{q}^{3n+1}$ such that $F|_{\xi}(\x)=0$ is an absolutely irreducible curve. Hence by Lemma \ref{leeyeo} the number $N'$ of non-singular points on it satisfies
\[ N' \ge q - \tfrac{1}{2}(d-1)(d-2)[2\sqrt{q}] + 1. \]
Note that $N' > 0$ provided that 
\[ q - (d-1)(d-2)\sqrt{q} + 1 > 0. \]
The above inequality clearly holds for $d=1,2$. Consequently if we assume $d \ne 1,2$ and take
\[ q > \tfrac{1}{4}(\alpha + \sqrt{\alpha^{2}-4})^{2}, \]
where $\alpha = (d-1)(d-2)$ then the curve $F|_{\xi}$ has a non-singular point. Finally it is not difficult to see that for $d \ne 1,2$
\[ 6d^{4}-8d^{3}+10d^{2} > (\alpha + \sqrt{\alpha^{2}-4})^{2}.\]
Hence $F(\x) = 0$ will always have a non-singular point whenever
\[ q > \tfrac{1}{2}(3d^{4}-4d^{3}+5d^{2}), \] 
as required.
\end{proof}
\vspace{0.4in}
\b{Acknowledgments:} It is a pleasure to thank the Hausdorff Research Institute for Mathematics, for their kind hospitality where part of this work took place. I would also like to thank Prof.\ Roger Heath-Brown for the valuable and interesting discussions I have had with him on this topic.
\newpage 

\bibliography{Non-singVarieties}
\bibliographystyle{plain}

\end{document}